\documentclass[10pt,a4paper,oneside,english]{amsart}
\usepackage[T1]{fontenc}
\usepackage[latin9]{inputenc}
\usepackage{units}
\usepackage{amsthm}
\usepackage{amssymb}

\makeatletter

\numberwithin{equation}{section}
\numberwithin{figure}{section}
\theoremstyle{plain}
\newtheorem{thm}{\protect\theoremname}[section]
  \theoremstyle{plain}
  \newtheorem{lem}[thm]{\protect\lemmaname}
  \theoremstyle{plain}
  \newtheorem{prop}[thm]{\protect\propositionname}
  \theoremstyle{remark}
  
  \theoremstyle{plain}
  \newtheorem{cor}[thm]{\protect\corollaryname}
  
\@ifundefined{date}{}{\date{}}

\usepackage{amsthm}

\usepackage{cite}

\newtheorem{theorem}{Theorem}[section]

\newtheorem{exa}[theorem]{Example}

\providecommand{\lemmaname}{Lemma}
  \providecommand{\propositionname}{Proposition}
  \providecommand{\remarkname}{Remark}
\providecommand{\theoremname}{Theorem}

\usepackage{babel}
\providecommand{\lemmaname}{Lemma}
  \providecommand{\propositionname}{Proposition}
  \providecommand{\remarkname}{Remark}
\providecommand{\theoremname}{Theorem}

\usepackage{babel}
\providecommand{\corollaryname}{Corollary}
  \providecommand{\lemmaname}{Lemma}
  \providecommand{\propositionname}{Proposition}
  \providecommand{\remarkname}{Remark}
\providecommand{\theoremname}{Theorem}

\makeatother

\usepackage{babel}
  \providecommand{\corollaryname}{Corollary}
  \providecommand{\lemmaname}{Lemma}
  \providecommand{\propositionname}{Proposition}
  \providecommand{\remarkname}{Remark}
\providecommand{\theoremname}{Theorem}

\def\z{
\par
\[
\phi(x)=\left(\left(1+\left|\alpha \right|_{\ell_1}  \right) x(1),\ldots,\left(1+\left|\alpha \right|_{\ell_1}  \right)x(n),x(n+1)-\sum_{i=1}^{n}\alpha(i)x(i),x(n+2)-\sum_{i=1}^{n}\alpha(i)x(i),\ldots \right).
\]}

\begin{document}

\title[Stability constants of the weak$^*$-fpp for the space $\ell_1$]{Stability constants of the weak$^*$ fixed point property for the space $\ell_1$ }

\author[E. Casini]{Emanuele Casini}
\address{\textsc{Emanuele Casini}, \rm{Dipartimento di Scienza e Alta Tecnologia, Universit\`{a} dell'Insubria,
via Valleggio 11, 22100 Como, Italy.}}
\email{emanuele.casini@uninsunbria.it}

\author[E. Miglierina]{Enrico Miglierina}
\address{\textsc{Enrico Miglierina}, \rm{Dipartimento di Discipline Matematiche, Finanza Matematica ed Econometria,
Universit\`{a} Cattolica del Sacro Cuore, Via Necchi 9, 20123 Milano,
Italy.}}
\email{enrico.miglierina@unicatt.it}

\author[\L. Piasecki]{\L ukasz Piasecki}
\address{\textsc{\L ukasz Piasecki}, \rm{Instytut Matematyki,
Uniwersytet Marii Curie-Sk{\l}odowskiej, Pl. Marii Curie-Sk{\l}odowskiej 1, 20-031 Lublin, Poland.}}
\email{piasecki@hektor.umcs.lublin.pl}

\author[R. Popescu]{Roxana Popescu}
\address{\textsc{Roxana Popescu}, \rm{Department of Mathematics, University of Pittsburgh, Pittsburgh, PA 15260,
	USA.}}
\email{rop42@pitt.edu}

\date{\today}

\numberwithin{equation}{section}

\begin{abstract}
The main aim of the paper is to study some quantitative aspects of the stability of the weak$^*$ fixed point property for nonexpansive maps in $\ell_1$ (shortly, $w^*$-fpp). We focus on two complementary approaches to this topic. First, given a predual $X$ of $\ell_1$ such that the $\sigma(\ell_1,X)$-fpp holds, we precisely establish how far, with respect to the Banach-Mazur distance, we can move from $X$ without losing the $w^*$-fpp. The interesting point to note here is that our estimate depends only on the smallest radius of the ball in $\ell_1$ containing all $\sigma(\ell_1,X)$-cluster points of the extreme points of the unit ball. Second, we pass to consider the stability of the $w^*$-fpp in the restricted framework of preduals of $\ell_1$. Namely, we show that every predual $X$ of $\ell_1$ with a distance from $c_0$ strictly less than $3$, induces a weak$^*$ topology on $\ell_1$ such that the $\sigma(\ell_1,X)$-fpp holds.

\end{abstract}

\keywords{weak$^*$ fixed point property, stability of weak$^*$ fixed point property, Lindenstrauss spaces, $\ell_1$ space, renorming}



\subjclass[2010]{47H10, 46B45, 46B25}


\maketitle
\section{Introduction}

Let $X$ be an infinite dimensional real Banach space and let us denote by $B_X$ its closed unit
ball and by $S_X$ its unit sphere.
A nonempty bounded closed and convex subset $C$ of $X$ has the fixed point property
(shortly, fpp) if each nonexpansive mapping (i.e., the mapping $T:C\rightarrow C$ such that $\|T(x)-T(y)\|\leq \|x-y\|$ for all $x,y\in C$) has a fixed point.
A dual space $X^*$ is said to have the $\sigma(X^*,X)$-fixed
point property ($\sigma(X^*,X)$-fpp or, shortly, $w^*$-fpp when no confusion can arise) if every nonempty, convex,
$\sigma(X^*,X)$-compact subset $C$ of $X^*$ has the fpp. 

The study of the
$\sigma(X^*,X)$-fpp reveals to be of special interest whenever a
dual space has different preduals. 
For
instance, this situation occurs when we consider the space $\ell_1$
and its preduals $c_0$ and $c$ where it is well-known (see
\cite{Karlovitz1976}) that $\ell_1$ has the $\sigma(\ell_1,c_0)$-fpp
whereas it lacks the $\sigma(\ell_1,c)$-fpp.
Necessary and sufficient conditions for a predual $X$ of $ \ell_1$ to be a space such that the $\sigma(\ell_1,X)$-fpp holds are proved in \cite{Casini-Miglierina-Piasecki2015}.
The present paper concerns the investigation of the stability of the $\sigma(\ell_1,X)$-fpp carrying on the study developed in \cite{CMPP4}. 
We recall that stability of fixed point property deals with the following question: let us suppose that a Banach space $X$ has the fixed point property and $Y$ is a Banach space isomorphic to $X$ with "small" Banach-Mazur distance, does $Y$ have fixed point property? This problem has been widely studied for fpp and only occasionally for weak$^*$ topology (see \cite{Soardi,Dominguez-Garcia-Japon1998}).
In \cite{CMPP4}, we established a characterization of stability of the $\sigma (\ell_1,X)$-fpp by means of a geometrical property. In the present paper we wish to investigate some quantitative aspects of stability of the $\sigma(\ell_1,X)$-fpp.  To this aim, when $X^*$ enjoys the $\sigma(X^*, X)$-fpp, it reveals to be useful to introduce the stability constant:
$$
\gamma^*(X)=\sup \left\lbrace \gamma\geq 1: Y^* \textrm{ has the } \sigma(Y^*,Y) \textrm{-fpp whenever }  d(X,Y)\leq \gamma\right\rbrace, 
$$
where $d(X,Y)$ denotes the Banach-Mazur distance between $X$ and $Y$.
Under the assumption that $X$ is a predual of $\ell_1$, Section \ref{Sec Stability} establishes the relation between the stability constant and the constant 
$$
r^*(X)=\inf \left\{r>0: \left(\textrm{ext}(B_{\ell_1})\right)' \subset r B_{\ell_1} \right\}
$$
where $\left(\textrm{ext}(B_{\ell_1})\right)'$ denotes the set of $\sigma(\ell_1, X)$-limit points of the extreme points of $B_{\ell_1}$. 
It is worth pointing out that the spaces satisfying $\left(\textrm{ext}(B_{X^*})\right)' \subset r B_{X^*}$ with $r<1$ have already been studied in the framework of polyhedrality for Banach spaces (see \cite{Durier-Papini1993, Fonf-Vesely2004, Casini-Miglierina-Piasecki-Vesely2016, CMPP4}). Moreover, it is known that if $r^*(X)=0$, then $X=c_0$ (see \cite{Durier-Papini1993}) and in this case $\gamma^*(c_0)\leq 2$ by a result by Lim \cite{Lim1980}. A lower bound for $\gamma^*(X)$ can be obtained from the key result of \cite{CMPP4}. Indeed, in that paper we proved that if $r^*(X) \in [0,1]$, then
$$
\gamma^*(X)\geq \dfrac{2}{1+r^*(X)}.
$$
Moreover, from Theorem 3.4 in \cite{CMPP4} we know that if $r^*(X)=1$, then $\gamma^*(X)=1$. The main result of Section \ref{Sec Stability} shows that for $r^*(X) \in (0,1)$ also the reverse inequality holds true. Therefore, we obtain that
$$
\gamma^*(X)= \dfrac{2}{1+r^*(X)}.
$$

The stability constant $\gamma^*(X)$ takes into account a broad family of perturbations of the space $X$. Indeed, if $Y$ is isomorphic to $X$, then both the classes of $\sigma(Y^*,Y)$-compact convex sets and of nonexpansive mappings on $Y^*$ are different from those in $X^*$. It may be interesting to deal with the stability of the $\sigma(\ell_1,X)$-fpp in the restricted framework of the preduals of $\ell_1$, hence by allowing only modifications to the class of $\sigma(Y^*,Y)$-compact convex sets. Namely, in Section \ref{Sec stability restricted}, we deal with the estimation of the following constant:
$$
\eta^*(X)=\sup \left\lbrace\eta \geq 1:\,Y^*=\ell_1,\,d(X,Y) \leq \eta \Rightarrow Y^*\,\textrm{ has the }\,\sigma(\ell_1,Y)\textrm{-fpp}\right\rbrace.     
$$ 
where $X$ is a predual of $\ell_1$ that enjoys the $\sigma(\ell_1,X)$-fpp. The main result of this section gives a sharp estimation of $\eta^*(c_0)$. Indeed, we prove that $\eta^*(c_0)=3$.
The estimate we obtain is interesting since it shows that we should move as far as $c$ is from $c_0$ in order to lack the $w^*$-fpp. Finally, we provide an example of a predual $X$ of $\ell_1$ such that the  $\sigma(\ell_1, X)$-fpp holds, whereas the Banach-Mazur distance between $c_0$ and $X$ is equal to $3$.  

\section{A quantitative view on stability of weak$^*$ fixed point property in $\ell_1$}\label{Sec Stability}
This section is devoted to study the stability of the $\sigma(\ell_1,X)$-fpp. In particular we provide a sharp estimation of the constant $\gamma^*(X)$ as defined in the introduction. In order to prove the main result of this section we need some preliminary steps.
We begin with the following simple lemma. Probably it is well known, but we were not able to find a suitable reference.
\begin{lem}\label{approximation_lemma}
Let ${x^*_n}\subset X^*$ be a sequence norm convergent to $x^*$. Then 
$$
\lim_{n \rightarrow \infty} d(\ker x^*, \ker x^*_n)=1.
$$
\end{lem}
\begin{proof}
Let us consider  a projection $P:X\rightarrow \ker x^*$. It is well known that ${P(x)=x-x^*(x)z}$, where $x^*(z)=1$. We can assume that $x_n^*(z)>0$. Then, let $\lambda_n \in \mathbb{R}$  be such that $x^*_n(\lambda_nz)=1$ for every $n$. It holds
$$
\left| 1-\lambda_n\right| =\left| x^*_n(\lambda_nz)-x^*(\lambda_nz) \right| \leq \left|\lambda_n\right| \|z\|\,\|x^*-x^*_n\|.
$$
Hence, $\lim_{n \rightarrow \infty} \lambda_n=1$. Now, we define a sequence of projections ${P_n:X\rightarrow \ker x^*_n}$ such that $P_n(x)=x-\lambda_nx^*_n(x)z$. Since
$$
\|P(x)-P_n(x)\|\leq \|\lambda_nx^*_n-x^*\|\,\|x\|\,\|z\|,
$$
we have $\lim_{n \rightarrow \infty} \|P-P_n\|=0$.
Now we recall the definition of operator opening between two closed subspaces $Y$ and $Z$ of the space $X$ (see \cite{Ostrovskii1994}). Let us denote by
$$
r_0(Y,Z)=\inf \left\lbrace\|\psi-I\|:\psi(Y)=Z,\,\psi \mbox{ invertible linear operator on } X \right\rbrace, 
$$
then the operator opening between $Y$ and $Z$ is defined by $r(Y,Z)=\max\left\lbrace r_0(Y,Z),r_0(Z,Y) \right\rbrace$. Theorem 4.2 (e) in \cite{Ostrovskii1994} implies
$$
\lim_{n \rightarrow \infty} r(\ker x^*,\ker x^*_n)=0.
$$ 
Finally, Proposition 6.1 in \cite{Ostrovskii1994} allows us to conclude that 
$$
\lim_{n \rightarrow \infty} d(\ker x^*,\ker x^*_n)=1.
$$ 
\end{proof}

Now, we pass to consider the properties of a suitable renorming of some hyperplanes of $c$, the space of convergent sequences. We recall that the standard duality between $c$ and $\ell_1$ is defined by
$$
x^*(x)=x^{*}(1)\lim_i x(i) + \sum_{j=1}^{\infty}x^{*}(j+1)x(j)
$$
for every $x=(x(1),x(2),\dots)\in c$ and $x^*=(x^*(1),x^*(2), \dots)\in \ell_1$. 

For every $\alpha=(\alpha(1),\alpha(2),\dots)\in B_{\ell_1}$, we define the following hyperplane of $c$:
$$
W_\alpha=\left\lbrace x=(x(1),x(2),\dots)\in c: \lim_{i \rightarrow \infty}x(i)=\sum_{i=1}^{+\infty}\alpha(i)x(i) \right\rbrace.
$$ 
For a detailed study of this class of spaces we refer the reader to \cite{Casini-Miglierina-Piasecki2014} and \cite{Casini-Miglierina-Piasecki2015}. Here we recall only that $W_\alpha^*=\ell_1$.

Here and subsequently we adopt the following notations: $\left\|\cdot\right\|_{\infty}$ and $\left|\cdot\right|_{\ell_1}$ denote respectively the standard norm in $c$ and $\ell_1$. Let $n\in\mathbb{N}$. For each sequence $x=(x(1),x(2),\dots)$ we put $P_n(x)=(x(1),\dots,x(n),0,0,\dots)$ and $R_n x=x-P_n x$, and $x^+$ and $x^-$ denotes the positive and negative part of $x$, respectively.

Propositions \ref{main prop 1} and \ref{main prop 2} presented below form one of the main steps in the proof of the main result of this section.

\begin{prop} \label{main prop 1} Let $e^*=(e^*(1),\dots,e^*(n),0,0,\dots)\in \ell_1$ with $r_n:=\left|e^*\right|_{\ell_1}\in (0,1)$.
 For all $x\in W_{e^*}$, define

\begin{equation*}
\left\|x\right\|_n=\left(\left\|R_n x^+\right\|_\infty\vee
r_n\left\|R_n x^-\right\|_\infty
+\left\|R_n x^-\right\|_\infty\vee
r_n\left\|R_n x^+\right\|_\infty\right)\vee (1+r_n)\left\|P_n x\right\|_\infty.
\end{equation*}
Then
\begin{equation*}
(W_{e^*},\left\|\cdot\right\|_n)^*=(\ell_1,\left|\cdot\right|_n),
\end{equation*}
where
\begin{equation*}
\left|f\right|_n=\max\left\{\frac{r_n \left|R_n f^+\right|_{\ell_1} + \left|R_n f^-\right|_{\ell_1} }{1+r_n},
\frac{\left|R_n f^+\right|_{\ell_1} + r_n\left|R_n f^-\right|_{\ell_1}}{1+r_n}\right\}+\frac{\left|P_n f\right|_{\ell_1}}{1+r_n},
\end{equation*}
and a duality map $\phi:\ell_{1}\rightarrow W_{e^*}^{*}$ is
defined by 
\[
(\phi(f))(x)=\sum_{j=1}^{+\infty}x(j)f(j),
\]
where $f=(f(1),f(2),\dots)\in\ell_{1}$ and $x=(x(1),x(2),\dots)\in W_{e^*}$.

\end{prop}

\begin{proof}
First, we begin by noticing that $\left\|\cdot\right\|_n$ is a norm equivalent to the  $\left\|\cdot\right\|_\infty$ norm,

\begin{equation*}
(1+r_n)\left\|x\right\|_\infty
=(1+r_n)\left(\left\|R_n x\right\|_\infty\vee \left\|P_n x\right\|_{\infty}\right)\leq
\left\|x\right\|_n  \leq 2\left(\left\|R_n x\right\|_\infty \vee
\left\|P_n x\right\|_\infty\right)=2\left\|x\right\|_\infty
\end{equation*}
so, from Theorem 4.3 in \cite{Casini-Miglierina-Piasecki2014}, we know that the dual of $(W_{e^*},\left\|\cdot\right\|_n)$ is representable by $\ell_1$ with duality map $\phi:\ell_{1}\rightarrow W_{e^*}$
defined by 
\[
(\phi(f))(x)=\sum_{j=1}^{+\infty}x(j)f(j),
\]
where $f=(f(1),f(2),\dots)\in\ell_{1}$ and $x=(x(1),x(2),\dots)\in W_{e^*}$.
Therefore it suffices to show that

\begin{equation*}
\left|f\right|_n=\sup\left\{\sum_{i=1}^{\infty}x(i) f(i): x\in
W_{e^*},\left\|x\right\|_n\leq 1 \right\}
\end{equation*}
for each $f\in\ell_1$.
As in \cite{Lim1980}, the supremum can be taken again over $x$ satisfying
$x(i) f(i)\geq  0$. In case $x(i) f(i)<0$, replace it by $0$ when
estimating from above. Also, notice that

\begin{equation*}
\frac{1}{2}\left|f\right|_{\ell_1}=\frac{\left|R_n f\right|_{\ell_1}+\left|P_n f\right|_{\ell_1}}{2} \leq
\left|f\right|_n\leq\frac{\left|R_n f^+\right|_{\ell_1}+\left|R_n f^-\right|_{\ell_1}+\left|P_n f\right|_{\ell_1}}{1+r_n}
=\frac{1}{1+r_n}\left|f\right|_{\ell_1}.
\end{equation*}

Without loss of generality, one can assume that
$$\left|f\right|_n=\frac{r_n}{1+r_n}\left|R_n f^+\right|_{\ell_1}+\frac{1}{1+r_n}\left|R_n f^-\right|_{\ell_1}+\frac{1}{1+r_n}\left|P_n f\right|_{\ell_1},$$
and so

$$\left|R_n f^-\right|_{\ell_1}\geq \left|R_n f^+\right|_{\ell_1}\hspace{2mm}  (\heartsuit). $$

There are three cases to consider.

\textbf{Case 1.} Assume that

\begin{equation}\label{case 1.1}
\left\|R_n x^+\right\|_\infty\vee
r_n\left\|R_n x^-\right\|_\infty=\left|x^+(i)\right|
\end{equation}
and
\begin{equation}\label{case1.2}
\left\|R_n x^-\right\|_\infty\vee r_n\left\|R_n x^+\right\|_\infty=\left|x^-(j)\right|
\end{equation}
for some $i,j\geq n+1$. Then
 $$\left\|x\right\|_n=\left(\left|x^+(i)\right|+\left|x^-(j)\right|\right)\vee
(1+r_n)\left\|P_n x\right\|_\infty.$$

\textbf{SubCase 1.1.} If
$(1+r_n)\left\|P_n x\right\|_\infty\leq\left|x^+(i)\right|+\left|x^-(j)\right|$, then

 \begin{eqnarray*}
  (1+r_n)f(x)&=&(1+r_n)\sum_{k=1}^{n}x(k)f(k)+(1+r_n)\sum_{k=n+1}^{\infty}x(k)f(k)
 \\&\leq &\left(\left|x^+(i)\right|+\left|x^-(j)\right|\right)\left|P_n f\right|_{\ell_1}+(1+r_n)\left|x^+(i)\right|\left|R_n f^+\right|_{\ell_1}
 \\ && +(1+r_n)\left|x^-(j)\right|\left|R_n f^-\right|_{\ell_1}
 \\&\leq &\left(\left|x^+(i)\right|+\left|x^-(j)\right|\right)
 \left(\left|P_n f\right|_{\ell_1}+r_n\left|R_n f^+\right|_{\ell_1}+\left|R_n f^-\right|_{\ell_1}\right)
  \end{eqnarray*}
and the last inequality holds since $r_n\left|x^-(j)\right|\leq\left|x^+(i)\right|$ by (\ref{case 1.1}) and ${\left|R_n f^+\right|_{\ell_1}\leq\left|R_n f^-\right|_{\ell_1}}$ by $(\heartsuit)$. Thus, $f(x)\leq \left\|x\right\|_n\left|f\right|_n$.

 \textbf{SubCase 1.2.} If $\left|x^+(i)\right|+\left|x^-(j)\right|\leq (1+r_n)\left\|P_n x\right\|_\infty$, then from (\ref{case 1.1}) we obtain $r_n\left|x^-(j)\right|\leq \left|x^+(i)\right|$ and so $(1+r_n)\left|x^-(j)\right|\leq(1+r_n)\left\|P_n x\right\|_\infty$, or equivalently ${\left|x^-(j)\right|\leq\left\|P_n x\right\|_\infty}$. Now we have
 
 \begin{eqnarray*}
 f(x)&\leq&\sum_{k=1}^{n}\left|x(k)\right|\left|f(k)\right|+\left|x^+(i)\right|\left|R_n f^+\right|_{\ell_1}
 +\left|x^-(j)\right|\left|R_n f^-\right|_{\ell_1}
 \\&\leq& \left\|P_n x\right\|_\infty\left(\left|P_n f\right|_{\ell_1}
 +r_n\left|R_n f^+\right|_{\ell_1}+\left|R_n f^-\right|_{\ell_1}\right).
 \end{eqnarray*}
 This time the last inequality holds since $\left|x^+(i)\right|-r_n\left\|P_n x\right\|_\infty\leq \left\|P_n x\right\|_\infty-\left|x^-(j)\right|$,
 $0\leq \left\|P_n x\right\|_\infty-\left|x^-(j)\right|$ and $\left|R_n f^+\right|_{\ell_1}\leq\left|R_n f^-\right|_{\ell_1}$ by $(\heartsuit)$. Therefore, we obtain again that $f(x)\leq \left\|x\right\|_n\left|f\right|_n$.

\textbf{Case 2.} $\left\|R_n x^+\right\|_\infty\vee r_n\left\|R_n x^-\right\|_\infty=\left|x^+(i)\right|$, $\left\|R_n x^-\right\|_\infty\vee r_n\left\|R_n x^+\right\|_\infty
=r_n\left|x^+(i)\right|$ for some $i\geq n+1$ and so $\left\|x\right\|_n=(1+r_n)\left|x^+(i)\right| \vee (1+r_n)\left\|P_n x\right\|_\infty=(1+r_n)\left\|x\right\|_{\infty}$. This further implies $\left\|R_n x^-\right\|_\infty\leq r_n\left\|x\right\|_{\infty}$ and so 
 
 \begin{eqnarray*}
 f(x)&\leq&\left\|x\right\|_{\infty}\left|P_n f\right|_{\ell_1}+\left\|x\right\|_{\infty}\left|R_n f^+\right|_{\ell_1}+r_n\left\|x\right\|_{\infty}\left|R_n f^-\right|_{\ell_1}
 \\&\leq&\left\|x\right\|_{\infty}\left|P_n f\right|_{\ell_1}+r_n\left\|x\right\|_{\infty}\left|R_n f^+\right|_{\ell_1}+\left\|x\right\|_{\infty}\left|R_n f^-\right|_{\ell_1}
 \\&=&\left\|x\right\|_{\infty}(\left|P_n f\right|_{\ell_1}+r_n\left|R_n f^+\right|_{\ell_1}+\left|R_n f^-\right|_{\ell_1}),
 \end{eqnarray*}
 where last inequality holds by $(\heartsuit)$. Thus, $f(x)\leq \left\|x\right\|_n\left|f\right|_n$ and Case 2 is completed.

 \textbf{Case 3.} $\left\|R_n x^+\right\|_\infty\vee r_n\left\|R_n x^-\right\|_\infty=r_n\left|x^-(j)\right|$, $\left\|R_n x^-\right\|_\infty\vee r_n\left\|R_n x^+\right\|_\infty
 =\left|x^-(j)\right|$ for some $j\geq n+1$ and so $\left\|x\right\|_n=(1+r_n)\left|x^-(j)\right|\vee (1+r_n)\left\|P_n x\right\|_\infty=(1+r_n)\left\|x\right\|_{\infty}$. Case 3 can be solved using similar ideas as in Case 2.

Since the set of points considered in the above cases forms a dense set in $W_{e^*}$, we have shown that

\begin{equation*}
\sup\left\{\sum_{i=1}^{\infty}x(i)f(i): x\in W_{e^*},\left\|x\right\|_n\leq 1 \right\}\leq \left|f\right|_n. 
\end{equation*}

 To prove the reversed inequality, one can consider a sequence $\left\{ x^N \right\}_{N>n+1}^{\infty} \subset W_{e^*}$ defined as follows:
 
  \begin{equation*}
    x^N(k)=
        \begin{cases}
           (\textrm{sgn} f_k)\frac{1}{1+r_n}  & \quad\text{for}\quad k=1,\dots,n,\\
           \frac{-1}{1+r_n} & \quad\text{for}\quad f(k)<0 \quad \text{and}\quad k=n+1,\dots,N,\\
            \frac{r_n}{1+r_n} & \quad\text{for}\quad f(k)\geq0 \quad \text{and}\quad k=n+1,\dots,N,\\
            \frac{1}{1+r_n} \sum\limits_{i=1}^{n} (\textrm{sgn} f_i)e^*(i) & \quad\text{for}\quad k\geq N+1.
        \end{cases}
  \end{equation*}

 \end{proof}

\begin{prop}\label{main prop 2}
$(W_{e^*},\left\|\cdot\right\|_n)^*=(\ell_1,\left|\cdot\right|_n)$ fails the $w^*$-fpp.
\end{prop}

\begin{proof} Let $C\subset \ell_1$ be defined by $$C=\left\{t_0 e^*+ \sum_{k=1}^{\infty}t_k e_{n+k}^* : t_i\geq 0,\sum_{i=0}^{\infty}t_i=1\right\}.$$
The set $C$ is convex and weak$^*$ compact in $(W_{e^*},\left\|\cdot\right\|_n)^*=(\ell_1,\left|\cdot\right|_n)$.
Consider a mapping $T:C\rightarrow C$  given by
$$T\left(t_0 e^*+ \sum_{k=1}^{\infty}t_k e_{n+k}^*\right)=\sum_{k=0}^{\infty}t_k e_{n+k+1}^*.$$


\noindent The map $T$ is fixed point free and $\left|\cdot\right|_n$-nonexpansive. Indeed, let $$t=(t_0 e^*(1),\dots,t_0 e^*(n),t_1,t_2,\dots) \textrm{ and } s=(s_0 e^*(1),\dots,s_0 e^*(n),s_1,s_2,\dots)$$  be two elements of the set $C$. We consider two cases:

 \textbf{Case 1}: $t_0-s_0\geq 0$.

 This further implies $\left|R_n(t-s)^-\right|_{\ell_1}\geq\left|R_n(t-s)^+\right|_{\ell_1}$ and so

\begin{equation*}
\left|t-s\right|_n
=\frac{r_n}{1+r_n}\left|R_n(t-s)^{+}\right|_{\ell_1}+\frac{1}{1+r_n}\left|R_n(t-s)^-\right|_{\ell_1}+\frac{r_n}{1+r_n}\left|t_0-s_0\right|.
\end{equation*}
Now
 \begin{eqnarray*}
 \left|T(t)-T(s)\right|_n
 &=& \max\left\{\frac{r_n}{1+r_n}(\left|t_0-s_0\right|+\left|R_n(t-s)^+\right|_{\ell_1})+\frac{1}{1+r_n}\left|R_n(t-s)^-\right|_{\ell_1}, \right.
 \\&&\left. \quad \quad \frac{1}{1+r_n}(\left|t_0-s_0\right|+\left|R_n(t-s)^+\right|_{\ell_1})+\frac{r_n}{1+r_n}\left|R_n(t-s)^-\right|_{\ell_1}\right\}
\\&= &\frac{r_n}{1+r_n}\left|R_n(t-s)^+\right|_{\ell_1}+\frac{1}{1+r_n}\left|R_n(t-s)^-\right|_{\ell_1}+\frac{r_n}{1+r_n}\left|t_0-s_0\right|
\\&=&\left|t-s\right|_n
 \end{eqnarray*}
and so $T$ is $\left|\cdot\right|_n$-isometry.

\textbf{Case 2}: $t_0-s_0\leq 0$. The proof is similar with Case 1.

\end{proof}

\begin{prop}\label{main prop 3}
If $X$ is a predual of $\ell_1$ with $r^*(X) \in (0,1)$, then $\gamma^*(X)\leq\frac{2}{1+r^*(X)}$.  
\end{prop}

\begin{proof}
Let $\epsilon \in (0,r^*(X))$ be arbitrarily chosen. There exist $e^* \in (\textrm{ext}(B_{\ell_1}))'$ and a subsequence $(e^*_{n_k})_{k\geq1}$ of the standard basis in $\ell_1$ such that $1>\left|e^*\right|_{\ell_1}>r^*(X)-\frac{\epsilon}{2}$,  $(e^*_{n_k})$ is $\sigma(\ell_1, X)$-convergent 
to $e^*$ and $\left|e^*\right|_{\ell_1}>\sum_{k=1}^{\infty}\left|e^*(n_k)\right|$.

\textbf{Step 1.} (Passing from $X$ to a hyperplane in $c$). Let $Z=\left[\left\{e_0^*, e_{n_1}^*, e_{n_2}^*,\dots\right\}\right]$ be the norm-closed linear span of $\left\{e_0^*, e_{n_1}^*, e_{n_2}^*,\dots\right\}$, where $$e_0^*=\frac{e^*-\sum_{k=1}^{\infty}e^*(n_k)e^*_{n_k}}{\left|e^*\right|_{\ell_1}-\sum_{k=1}^{\infty}\left|e^*(n_k)\right|}.$$ Since $\overline{\left\{e_0^*, e_{n_1}^*, e_{n_2}^*,\dots\right\}}^{\textrm{w}^*}=\left\{e_0^*, e_{n_1}^*, e_{n_2}^*,\dots\right\}\cup \left\{e^*\right\}\subset Z$, Lemma 1 in
\cite{A1992} assures that $\overline{\left[\left\{e_0^*, e_{n_1}^*, e_{n_2}^*,\dots\right\}\right]}^{\textrm{w}^*}=Z$. Thus $Z=(X/^{\bot}Z)^*$. Let $y^*\in \ell_1$ be defined as
\[
y^*=\left(\left|e^*\right|_{\ell_1}-\sum_{k=1}^{\infty}\left|e^*(n_k)\right|,e^*(n_1),
e^*(n_2), e^*(n_3),\dots \right).
\]
Since $y^* \in B_{\ell_1}$, we know that $W_{y^*}^*=\ell_1$ and $
y_{n}^*\overset{\sigma(\ell_1,W_{y^*})}{\longrightarrow}y^*$, where $(y_n^*)$ denotes the standard basis in $\ell_1$.
Let $\phi$ be the basis to basis map of $Z$ onto $\ell_1=W_{y^*}^*$, that is, $\phi\left(a_1
e_0^*+a_2 e_{n_1}^*+a_3 e_{n_2}^*+a_4 e_{n_3}^*+\dots \right)=\sum_{k=1}^{\infty}a_k y_k^*.$ Then $\phi(e^*)=y^*$. Indeed,
\begin{eqnarray*}
\phi(e^*)&=&\phi \left(\left(\left|e^*\right|_{\ell_1}-\sum_{k=1}^{\infty}\left|e^*(n_k)\right|\right)e_{0}^*+ \sum_{k=1}^{\infty} e^*(n_k)
e^*_{n_k}\right)
\\&=&\left(\left|e^*\right|_{\ell_1}-\sum_{k=1}^{\infty}\left|e^*(n_k)\right|\right)
y_1^*+ \sum_{k=1}^{\infty} e^*(n_k) y^*_{k+1} 
\\&=&
\left(\left|e^*\right|_{\ell_1}-\sum_{k=1}^{\infty}\left|e^*(n_k)\right|,e^*(n_1),e^*(n_2),\dots\right)=y^*.
\end{eqnarray*}
Consequently, $\phi$ is a $w^*$-continuous homeomorphism from
$\overline{\left\{e_0^*,e_{n_1}^*,e_{n_2}^*,\dots\right\}}^{\textrm{w}^*}=\left\{e_0^*,e_{n_1}^*,e_{n_2}^*,\dots\right\} \cup \left\{e^*\right\}$
onto $\overline{\left\{y_1^*,y_2^*,\dots\right\}}^{\textrm{w}^*}=\left\{y_1^*,y_2^*,\dots\right\} \cup \left\{y^*\right\}$. By Lemma 2
in \cite{A1992} we see that $\phi$ is a $w^*$-continuous isometry
from $Z$ onto $\ell_1=W_{y^*}^*$. Therefore $W_{y^*}$ is isometric to $X/^{\bot}Z$.
 Since $\lim_{n}\left|y^*-P_{n}y^*\right|_{\ell_1}=0$, Theorem 4.3 in \cite{Casini-Miglierina-Piasecki2014} and Lemma \ref{approximation_lemma} assure that there are $n\in \mathbb{N}$ and an isomorphism $\psi:W_{y^*}\rightarrow W_{P_{n}y^*}$ such that $\left|y^*-P_{n}y^*\right|_{\ell_1}\leq \frac{\epsilon}{2}$ and $(1-\epsilon)\left\|x\right\|_{\infty}\leq\left\|\psi(x)\right\|_{\infty} \leq (1+\epsilon)\left\|x\right\|_{\infty}$ for all $x \in W_{y^*}$.

\textbf{Step 2.} (Renorming of a hyperplane in $c$). Put $r_n:=\left|P_{n}y^*\right|_{\ell_1}$. As in Propositions \ref{main prop 1} and \ref{main prop 2}, we consider a renorming defined for all $x\in W_{P_{n}y^*}$ by

\begin{equation*}
\left\|x\right\|_n=\left(\left\|R_{n}x^+\right\|_\infty\vee
r_n\left\|R_{n}x^-\right\|_\infty
+\left\|R_{n}x^-\right\|_\infty\vee
r_n \left\|R_{n}x^+\right\|_\infty\right)\vee (1+r_n)\left\|P_{n}x\right\|_{\infty},
\end{equation*}
a set $C\subset \ell_1$ given by $$C=\left\{t_0 P_ny^*+ \sum_{k=1}^{\infty}t_k e_{n+k}^* : t_i\geq 0,\sum_{i=0}^{\infty}t_i=1\right\},$$
and a  $\left|\cdot\right|_n$-nonexpansive fixed point free mapping $T:C\rightarrow C$  defined by
$$T\left(t_0 P_ny^*+ \sum_{k=1}^{\infty}t_k e_{n+k}^*\right)=\sum_{k=0}^{\infty}t_k e_{n+k+1}^*.$$
Let $I$ from $\left(\ell_1, \left|\cdot \right|_n  \right)=\left( W_{P_n y^*},\left\|\cdot \right\|_n \right)^*$ onto $\left(\ell_1, \left|\cdot \right|_{\ell_1}  \right)=\left( W_{P_n y^*},\left\|\cdot \right\|_\infty \right)^*$ be the identity map.

\textbf{Step 3.} (Back to $X$). The mapping $\phi^{-1}\psi^* I$ is a weak$^*$ continuous isomorphism from $(\ell_1, \left|\cdot \right|_n)$ onto $(Z,\left|\cdot \right|_{\ell_1})$ and satisfies
$$(1+r_n)(1-\epsilon)\left|x\right|_n \leq \left|\phi^{-1}\psi^* Ix\right|_{\ell_1}\leq2(1+\epsilon)\left|x\right|_n$$
for all $x \in (\ell_1, \left|\cdot \right|_n)$. Taking also into account the following estimate
\begin{eqnarray*}
r_n&=& \left|P_ny^*\right|_{\ell_1}\geq\left|y^*\right|_{\ell_1}-\left|P_ny^*-y^*\right|_{\ell_1}=\left|e^*\right|_{\ell_1}-\left|P_ny^*-y^*\right|_{\ell_1}
\\&\geq& r^*(X)-\frac{\epsilon}{2}-\frac{\epsilon}{2}=r^*(X)-\epsilon>0,
\end{eqnarray*}
we conclude that
$$(1-\epsilon)(1+r^*(X)-\epsilon) (B_{X^*}\cap Z)\subset (\phi^{-1}\psi^* I)(B_{(\ell_1,\left|\cdot\right|_n)})\subset 2(1+\epsilon)(B_{X^*}\cap Z).$$

Next we define the set $D \subset \ell_1=X^*$ by
$$D=\textrm{conv}\left((1-\epsilon)(1+r^*(X)-\epsilon) B_{X^*}\cup (\phi^{-1}\psi^* I)(B_{(\ell_1,\left|\cdot\right|_n)})\right).$$
It is easy to check that $D$ is convex, symmetric, weak$^*$ compact,
and $0$ is its interior point. Therefore $D$ is a dual unit ball of an
equivalent norm $\left\|\left|\cdot\right|\right\|$ on $X$. Let
${Y=(X,\left\|\left|\cdot\right|\right\|)}$. Obviously, $D=B_{Y^*}$. Since $D\cap Z =(\phi^{-1}\psi^* I)(B_{(\ell_1,\left|\cdot\right|_n)})$, the mapping $\phi^{-1}\psi^* I$ is a weak$^*$ continuous isometry from $(\ell_1, \left|\cdot \right|_n)$ into $Y^*$. All the above implies that the set $(\phi^{-1}\psi^* I)(C)$ is convex, weak$^*$ compact in $Y^*$ and fails the fpp. Finally we observe that 
$$d(X,Y)\leq \frac{2(1+\epsilon)}{(1-\epsilon)(1+r^*(X)-\epsilon)}.$$
Therefore $\gamma^*(X)\leq\frac{2}{1+r^*(X)}$.

\end{proof}

Consequently, by combining remarks made in the introduction with the thesis of Proposition \ref{main prop 3}, we obtain the exact value of $\gamma^*(X)$.

\begin{thm}\label{main Thm gamma stability}
If $X$ is a predual of $\ell_1$ such that the $\sigma(\ell_1,X)$-fpp holds, then
$$
\gamma^*(X) = \frac{2}{1+r^*(X)}.
$$
\end{thm}
We recall that stability property for the $\sigma(\ell_1,c_0)$-fpp was already investigated in \cite{Soardi}. The important point to note here is the difference between our notion of stability and the approach developed in \cite{Soardi}. Indeed, in that paper the author considered only the $\sigma(\ell_1,c_0)$-topology on every renorming of $\ell_1$.

\section{Stability in the restricted framework of preduals of $\ell_1$: the case of $c_0$}\label{Sec stability restricted}

In this section we deal with the stability of the $w^*$-fpp in the restricted framework of Lindenstrauss spaces. Namely, for a predual $X$ of $\ell_1$ that enjoys the $\sigma(\ell_1,X)$-fpp, we are interested in the estimation of the following constant
$$
\eta^*(X)=\sup \left\lbrace\eta \geq 1:\,Y^*=\ell_1,\,d(X,Y) \leq \eta \Rightarrow Y^*\,\textrm{ has the }\,\sigma(\ell_1,Y)\textrm{-fpp}\right\rbrace.     
$$
We restrict our attention to the case when $X=c_0$.
The first step in this matter is the easy remark that $\eta^*(c_0)\leq 3$. Indeed, we know that $c$ does not satisfy the $\sigma (\ell_1,c)$-fpp and $d(c,c_0)=3$ by  the result in \cite{Cambern1968}.
Since $\gamma^*(c_0)\leq \eta^*(c_0)$, it follows from the previous section that $2\leq \eta^*(c_0) \leq 3$. In the sequel we prove that $\eta^*(c_0)=3$. 
In order to fix some notations we recall a well-known theorem.

\begin{thm} (see, e.g., \cite{Megg}). Let $T:X \rightarrow Y$ be a linear map from the Banach space $X$ onto the Banach space $Y$. Then there exists a linear map $\tilde{T}: X/\ker T \rightarrow Y$ such that
\begin{enumerate}
\item $\tilde{T}$ is an onto isomorphism,
\item $T=\tilde{T}\pi$, where $\pi:X\rightarrow X/\ker T$ denotes the quotient map,
\item $\| T \|=\| \tilde{T} \|$.
\end{enumerate}
\end{thm}

In order to prove the main theorem of this section we need some preliminary results.
 
First, we state the following lemma. The proof is easy and therefore we leave it to the reader.
\begin{lem}\label{Lemma Sec 3}
	Let $T:X \to Y$ be an onto bounded linear operator, where $Y\neq \left\lbrace 0\right\rbrace $. 
	Then $$ \sup \left\lbrace \delta >0: \delta B_Y \subseteq T(B_X) \right\rbrace =\| \tilde{T}^{-1}\|^{-1}.$$
\end{lem}


Moreover, the proof of our result relies on the following theorem by Alspach.

\begin{thm}(\cite{A1979})\label{Al}
	Le $X$ be a quotient of $c_0$. Then, for every $\epsilon >0$, there is a subspace $Y$ of $c_0$ such that $d(X,Y) < 1 + \epsilon$.
\end{thm}
Finally, the last result that plays a role in our argument can be viewed as an extension of Cambern's result (\cite{Cambern1968}) about the distance from $c_0$ to $c$.
\begin{prop}\label{Prop Sec 3}
	Let $X$ be a Banach space containing $c$ and let ${T:c_0 \to X}$ be an onto linear operator with $\| T \|=1$. Then $\| \tilde{T}^{-1}\| \geq 3$.
\end{prop}

\begin{proof}
 By Theorem \ref{Al} we know that there exists a subspace $Y$ of $c_0$ and an isomorphism ${S:c_0/\ker(T) \to Y}$ such that $\|S\| \, \| S^{-1}\| < 1+\epsilon$. Let $\tilde{T}:c_0/\ker(T) \to X$ and denote by $R$ the restriction of $S\tilde{T}^{-1} $ to $c$. Then $R$ is an isomorphism from $c$ into $c_0$ and by Theorem 2.1 in \cite{Gordon1970} it holds  
 $$3 \leq \|R \| \,\|R^{-1} \| \leq \|S\tilde{T}^{-1} \|\, \| (S\tilde{T}^{-1})^{-1}\| \leq \|\tilde{T}^{-1}\|\, \| S \|\, \| \tilde{T}\|\, \| S^{-1}\| \leq (1+\epsilon)\| \tilde{T}^{-1}\|.$$
Hence we have $\| \tilde{T}^{-1}\|\geq 3$.
\end{proof}
Now, we are in position to prove the main result of this section. It allows us to obtain the equality $\eta^* (c_0)=3$.
\begin{thm}
	Let $X$ be a predual of $\ell_1$ isomorphic to $c_0$. Suppose that $X^*$ fails the $w^*$-fpp. If $T: X \to c_0$ is an onto isomorphism with $\|T^{-1}\|=1$, then $\| T \| \geq 3$.
\end{thm}

\begin{proof}
	By Theorem 4.1 in \cite{Casini-Miglierina-Piasecki2015} there exists a quotient $X/Y$ of $X$ that is isometric to a space $W_\alpha$ with $|\alpha|_{\ell_1}=1$. Let us first suppose in addition that $W_\alpha$ contains a subspace isometric to $c$. Now, let us consider the onto map: $\pi T^{-1}: c_0 \to W_\alpha$, where $\pi:X\rightarrow X/Y$ is the quotient map. 
	It is easy to check that
	$\pi T^{-1}(B_{c_0}) \supseteq \frac{1}{\| T\|+\varepsilon}B _{W_\alpha}$, for every $\varepsilon >0$.
	Hence, by applying Lemma \ref{Lemma Sec 3}, we obtain
	$$\left\| \left( \widetilde{ \pi T^{-1}}\right)^{-1} \right\| ^{-1} \geq \dfrac{1}{\|T\|}.$$
	Moreover, since $\|\pi T^{-1} \|=1$,  Proposition \ref{Prop Sec 3} gives $\left\| \left( \widetilde{ \pi T^{-1}}\right)^{-1} \right\| \geq 3$. Therefore $\| T \| \geq 3$.
	
	Now we consider a space $W_\alpha$ such that $c$ is not included in. Let us denote  $\left(\alpha(1),\ldots,\alpha(n),\sum_{i=n+1}^{+\infty}\left|\alpha(i)\right|,0,0,\ldots \right)$ by $\alpha_n$. Then, clearly $\lim_{n \rightarrow \infty} |\alpha-\alpha_n|_{\ell_1}=0$. Moreover, by Proposition 2.1 in \cite{Casini-Miglierina-Piasecki2015}, the space $W_{\alpha_n}$ contains an isometric copy of $c$, for every $n$.  By Lemma \ref{approximation_lemma}, we have that $\lim_{n \rightarrow \infty} d(W_\alpha,W_{\alpha_n})=1$, which completes the proof.
	
\end{proof}

We recall that, as already said, the space $c$ is such that $\ell_1$ lacks the $\sigma(\ell_1, c)$-fpp and $d(c_0,c)=3$. Therefore, the equality $\eta^*(c_0)=3$ follows directly from the previous theorem.

On the other hand, we have an example of a predual $X$ of $\ell_1$ such that the $\sigma(\ell_1,X)$-fpp holds and $d(c_0,X)=3$. This example is based on a space belonging to the class of hyperplanes of $c$ already considered in Section \ref{Sec Stability}. 
We begin with a general result on all the hyperplanes $W_\alpha$ with $\alpha \in S_{\ell_1}$.
\begin{prop}\label{prop_dist_3}
If $\alpha \in S_{\ell_1}$, then $d(c_0,W_\alpha)=3$.
\end{prop} 
 \begin{proof}
Let us consider the space $W_\alpha$ where $\alpha=(\alpha(1),\ldots,\alpha(n),0,0,\ldots)\in B_{\ell_1}$ and let $\phi:W_\alpha\rightarrow c_0$ be defined by
{

\medmuskip=0mu
\thinmuskip=0mu
\thickmuskip=0mu

\z

\medmuskip=-2mu
\thinmuskip=-2mu
\thickmuskip=-2mu
\nulldelimiterspace=-1pt
\scriptspace=0pt
}
 
\noindent It is easy to check that $\phi$ is an onto isomorphism. Moreover, if $\alpha \in S_{\ell_1}$, then it holds
\begin{equation*}
\|\phi\|\,\|\phi^{-1}\| = 3
\end{equation*}
and the space $W_\alpha$ contains an isometric copy of $c$ (see Proposition 2.1 in \cite{Casini-Miglierina-Piasecki2015}). Hence, Theorem 2.1 in \cite{Gordon1970} implies that $ \|\varphi\|\,\|\varphi^{-1}\| \geq 3$ for every isomorphism $\varphi$ from $W_\alpha$ onto $c_0$. Therefore, $d(W_\alpha,c_0)=3$ for every $\alpha\in S_{\ell_1}$ with a finite number of non null components.
Now we pass to consider the general case. Let $\alpha \in S_{\ell_1}$ and $\alpha_n=\left(\alpha(1),\ldots,\alpha(n),\sum_{i=n+1}^{+\infty}\left|\alpha(i)\right|,0,0,\ldots \right)$. Then, clearly $\lim_{n \rightarrow \infty} |\alpha-\alpha_n|_{\ell_1}=0$. By Lemma \ref{approximation_lemma}, we have that $\lim_{n \rightarrow \infty} d(W_\alpha,W_{\alpha_n})=1$. Therefore, 
we conclude that $d(c_0,W_\alpha)=3$.
\end{proof} 
We are now in position to state the example announced above.
\begin{exa}
Let $\alpha = \left( -\frac{1}{2}, -\frac{1}{4}, -\frac{1}{8}, \cdots\right) \in S_{\ell_1}$. The space $W_\alpha$ is a predual of $\ell_1$ such that $d(c_0,W_\alpha)=3$ by Proposition \ref{prop_dist_3}. Moreover,
$\ell_1$ has the $\sigma(\ell_1,W_\alpha)$-fpp by Proposition 2.2 in \cite{Casini-Miglierina-Piasecki2015}. 
\end{exa}
By following the same arguments as in the proof of Proposition \ref{prop_dist_3} we obtain an upper bound for the estimation of the Banach-Mazur distance between one of the considered hyperplane and $c_0$. More precisely, we have
$$
d(W_\alpha,c_0)\leq1+2 |\alpha|_{\ell_1}
$$
for every $\alpha \in B_{\ell_1}$. The last inequality, combined with Theorem \ref{main Thm gamma stability}, gives the following characterization of stability of the $\sigma(\ell_1,W_\alpha)$-fpp in the sense of the constant $\gamma^*(X)$.
\begin{cor}
Let $\alpha \in B_{\ell_1}$ be such that $\sigma(\ell_1,W_\alpha)$-fpp holds. Then $\gamma^*(W_\alpha)>1$ if and only if $d(W_\alpha,c_0)<3$.
\end{cor}

\end{document}